\documentclass[12pt,reqno]{amsart}
\usepackage{amsmath, amsfonts, amssymb, amsthm, hyperref,cite}
\usepackage{bm}
\usepackage{mathrsfs}
\allowdisplaybreaks[4]
\def\l{\left}
\def\r{\right}
\def\bg{\bigg}
\def\({\bg(}
\def\){\bg)}
\def\t{\text}
\def\f{\frac}

\def\eq{\equiv}

\def\Z{\mathbb Z}
\def\C{\mathbb C}
\def\N{\mathbb N}

\def\<{\langle}
\def\>{\rangle}
\def\1{{\bf 1}}

\def\mod #1{{\rm mod}\ #1}

\theoremstyle{plain}
\newtheorem{theorem}{Theorem}
\newtheorem{conjecture}{Conjecture}

\newtheorem{lemma}{Lemma}
\newtheorem{corollary}{Corollary}
\theoremstyle{definition}

\newtheorem*{Acks}{Acknowledgments}
\theoremstyle{remark}
\newtheorem{Rem}{Remark}

\numberwithin{equation}{section}

\begin{document}
\title[Congruences concerning generalized central trinomial coefficients]{Congruences concerning generalized central \\trinomial coefficients}

\author[Jia-Yu Chen]{Jia-Yu Chen}
\address[Jia-Yu Chen]{Department of Mathematics, Nanjing
University, Nanjing 210093, People's Republic of China}
\email{563917224@qq.com}

\author[Chen Wang]{Chen Wang}
\address[Chen Wang]{Department of Mathematics, Nanjing
University, Nanjing 210093, People's Republic of China}
\email{cwang@smail.nju.edu.cn}

\begin{abstract}
For any $n\in\mathbb{N}=\{0,1,2,\ldots\}$ and $b,c\in\mathbb{Z}$, the generalized central trinomial coefficient $T_n(b,c)$ denotes the coefficient of $x^n$ in the expansion of $(x^2+bx+c)^n$. Let $p$ be an odd prime. In this paper, we determine the summation $\sum_{k=0}^{p-1}T_k(b,c)^2/m^k$ modulo $p^2$ for integers $m$ with certain restrictions. As applications, we confirm some conjectural congruences of Sun [Sci. China Math. 57 (2014), 1375--1400].
\end{abstract}

\subjclass[2020]{Primary 11A07, 11B75; Secondary 05A10, 11B65}
\keywords{Generalized central trinomial coefficients, binomial coefficients, congruences}
\thanks{The second author is the corresponding author}

\maketitle

\section{Introduction}
\setcounter{lemma}{0} \setcounter{theorem}{0}
\setcounter{equation}{0}\setcounter{proposition}{0}
\setcounter{Rem}{0}\setcounter{conjecture}{0}

For $n\in\N=\{0,1,2,\ldots\}$, the $n$th central trinomial coefficient $T_n$ is the coefficient of $x^n$ in the expansion of $(x^2+x+1)^n$. By the multinomial theorem, it is easy to see that
\begin{equation}\label{tnformula}
T_n=\sum_{k=0}^{\lfloor n/2\rfloor}\binom{n}{2k}\binom{2k}{k}=\sum_{k=0}^{\lfloor n/2\rfloor}\binom{n}{k}\binom{n-k}{k}.
\end{equation}
Central trinomial coefficients have some interesting combinatorial interpretations. For example, $T_n$ counts the lattice paths from $(0,0)$ to $(n,0)$ with steps $U=(1,1),\ D=(1,-1)$ and $H=(1,0)$. For more combinatorial interpretations and formulas of $T_n$, one may consult \cite[A002426]{OEIS}.

For $b,c\in\Z$, the generalized central trinomial coefficients $T_n(b,c)$ are defined as the coefficients of $x^n$ in the expansion of $(x^2+bx+c)^n$. By the multinomial theorem, we have the following formula similar to \label{tnformula}:
\begin{equation}\label{tnbcformula}
T_n(b,c)=\sum_{k=0}^{\lfloor n/2\rfloor}\binom{n}{2k}\binom{2k}{k}b^{n-2k}c^k=\sum_{k=0}^{\lfloor n/2\rfloor}\binom{n}{k}\binom{n-k}{k}b^{n-2k}c^k.
\end{equation}
Some well-known combinatorial sequences are the special cases of $T_n(b,c)$. For example,
$$
T_n=T_n(1,1),\ \binom{2n}{n}=T_n(2,1),\ D_n=T_n(3,2),
$$
where $D_n=\sum_{k=0}^n\binom{n}{k}\binom{n+k}{k}$ is the $n$th central Delannoy number (cf. \cite[A001850]{OEIS}). Throughout the paper, we set $d=b^2-4c$. The generalized central trinomial coefficients are also related to the well-known Legendre polynomials $P_n(x)=\sum_{k=0}^n\binom{n}{k}\binom{n+k}{k}((x-1)/2)^k$ (cf. \cite[p. 38]{Gould}) via the following identity (see \cite{Noe,Sun2014a,Sun2014b}):
$$
T_n(b,c)=(\sqrt{d})^n P_n\l(\f{b}{\sqrt{d}}\r).
$$

It is known that sums involving central binomial coefficients sometimes have beautiful congruence properties (cf. e.g., \cite{GuGuo2020,Liu2017,MaoPan2017,MT2013,MT2018,SunZH2011,Sun2010,Sun2011,Sun2011a}). As mentioned above, the central binomial coefficients $\binom{2n}{n}=T_n(2,1)$. Motivated by this, Sun \cite{Sun2014a,Sun2014b} investigated the congruences for sums involving generalized central trinomial coefficients systematically. For any odd prime $p$, Sun \cite{Sun2014a} proved some congruences modulo $p$ for sums of the following forms:
$$
\sum_{k=0}^{p-1}\binom{2k}{k}T_k(b,c)/m^k\quad\t{and}\quad\sum_{k=0}^{p-1}\binom{2k}{k}T_{2k}(b,c)/m^k.
$$
Moreover, he also posed some conjectures, one of which states as follows: for any prime $p>3$ we have
\begin{equation}\label{wangsunres}
\sum_{k=0}^{p-1}\f{\binom{2k}{k}}{12^k}T_k\eq\l(\f{p}3\r)\f{3^{p-1}+3}{4}\pmod{p^2},
\end{equation}
where $(\f{\cdot}{p})$ stands for the Legendre symbol. It is worth mentioning that \eqref{wangsunres} has been confirmed by the second author and Sun \cite{WangSun}.

Sun \cite{Sun2014b} obtained some congruences involving $T_n(b,c)^2$. For instance, he proved that for any $b,c\in\Z$ and prime $p>3$ with $p\nmid d$ we have
\begin{equation}\label{sunres}
\sum_{k=0}^{p-1}\f{T_k(b,c)^2}{d^k}\eq\l(\f{16c}{d}\r)^{(p-1)/2}+p\sum_{\substack{k=0\\ k\neq (p-1)/2}}^{p-1}\f{\binom{2k}{k}}{2k+1}\l(-\f{c}{d}\r)^k\pmod{p^3}.
\end{equation}
In particular, taking $b=2, c=-1$ in \eqref{sunres} and via some further computation, he deduced that
\begin{equation}\label{sunres1}
\sum_{k=0}^{p-1}\f{T_k(2,-1)^2}{8^k}\eq\l(\f{-2}{p}\r)\pmod{p^2}.
\end{equation}
Moreover, Sun \cite[Conjecture 5.4]{Sun2014b} posed the following conjecture.

\begin{conjecture}\label{sunconj}
Let $p$ be an odd prime. Then
\begin{equation}\label{sunconjeq1}
\sum_{k=0}^{p-1}\f{T_k(2,2)^2}{4^k}\eq\sum_{k=0}^{p-1}\f{\binom{2k}{k}^2}{8^k}\quad \l(\mod{p^{(5+(\f{-1}{p}))/2}}\r).
\end{equation}
If $p>3$, then
\begin{equation}\label{sunconjeq2}
\sum_{k=0}^{p-1}\f{T_k(4,1)^2}{4^k}\eq\sum_{k=0}^{p-1}\f{T_k(4,1)^2}{36^k}\eq\l(\f{-1}{p}\r)\pmod{p^2}.
\end{equation}
\end{conjecture}
Clearly, Conjecture \ref{sunconj} cannot be deduced from \eqref{sunres}.

Let us introduce some concepts which are necessary to state our main result. Let $p$ be a prime. Then for any integer $a\not\eq0\pmod{p}$ we have $a^{p-1}\eq1\pmod{p}$ by the Fermat's little theorem, and then $q_p(a):=(a^{p-1}-1)/p\in\Z_p$. As usual, we call $q_p(a)$ a Fermat quotient with base $a$. For any $p$-adic integer $x$, let $\<x\>_p$ denote the least nonnegative residue of $x$ modulo $p$. Set
$$
S_{p-1}(x):=\sum_{k=1}^{p-1}\f{\binom{2k}{k}}{k}x^k.
$$

Now we state our main theorem which gives different parametric congruences for the summation $\sum_{k=0}^{p-1}T_n(b,c)^2/m^k$.

\begin{theorem}\label{maintheorem}
Let $p$ be an odd prime and let $b,c\in\Z$ and $d=b^2-4c$.

{\rm (i)} If $p$ does not divide $d$, then we have
\begin{equation}\label{maintheq1}
\sum_{k=0}^{p-1}\f{T_k(b,c)^2}{(-d)^k}\eq\sum_{k=0}^{(p-1)/2}\binom{2l}{l}^2\l(-\f{c}{4d}\r)^l\pmod{p^2}.
\end{equation}

{\rm (ii)} Let $m$ be an integer satisfying the equation $(m-d)^2=16mc$. If $p\nmid md(m-d)$, then we have
\begin{align}\label{maintheq2}
&\sum_{k=0}^{p-1}\f{T_k(b,c)^2}{m^k}\notag\\
\eq&\l(\f{-1}{p}\r)+\f{pd}{d-m}\l(\f{-1}{p}\r)\l(q_p(d)-q_p(m)+S_{p-1}\l(\f{m+d}{4m}\r)-S_{p-1}\l(\f{m+d}{4d}\r)\r)\pmod{p^2}.
\end{align}
\end{theorem}

In fact, $S_{p-1}((m+d)/(4m))$ and $S_{p-1}((m+d)/(4d))$ modulo $p$ in \eqref{maintheq2} can be determined. If $p\mid m+d$, then the two sums are both equivalent to $0$ modulo $p$. Now suppose that $p\nmid m+d$. Sun and Tauraso \cite[Theorem 1.2]{Sun2010} proved that for any prime $p$ and integer $t$ with $p\nmid t$ we have
\begin{equation}\label{suntaurasores}
\sum_{k=1}^{p-1}\f{\binom{2k}{k}}{k(-t)^k}\eq \f{2t^p-2V_p(t)}{pt}\pmod{p},
\end{equation}
where the polynomial sequence $\{V_n(x)\}_{n\in\N}$ is defined as follows:
$$
V_0(x)=2,\ V_1(x)=x,\ \t{and}\ V_{n+1}(x)=x(V_n(x)+V_{n-1}(x))\ (n\in\Z^{+}).
$$
Sun and Tauraso also showed that $V_p(t)\eq t\pmod{p}$. Thus the right-hand side of \eqref{suntaurasores} are always $p$-adic integers. By \eqref{suntaurasores} we obtain
$$
S_{p-1}\l(\f{m+d}{4m}\r)\eq \f{2\l\<-\f{4m}{m+d}\r\>_p^p-2V_p\l(\l\<-\f{4m}{m+d}\r\>_p\r)}{p\l\<-\f{4m}{m+d}\r\>_p}\pmod{p}
$$
and
$$
S_{p-1}\l(\f{m+d}{4d}\r)\eq \f{2\l\<-\f{4d}{m+d}\r\>_p^p-2V_p\l(\l\<-\f{4d}{m+d}\r\>_p\r)}{p\l\<-\f{4d}{m+d}\r\>_p}\pmod{p}.
$$

By Theorem \ref{maintheorem}, we can obtain the following results. (The reader may try to find more applications of Theorem \ref{maintheorem}.)
\begin{corollary}\label{cor}
Let $p$ be an odd prime. Then modulo $p^2$ we have
\begin{align}\label{coreq1}
&\sum_{k=0}^{p-1}\f{T_k(2,2)^2}{4^k}\eq\sum_{k=0}^{p-1}\f{\binom{2k}{k}^2}{8^k}\notag\\
\eq&\begin{cases}\displaystyle\l(\f{2}{p}\r)\l(2x-\f{p}{2x}\r)\ &if\ p\eq1\pmod{4}\ and\ p=x^2+4y^2\ (x,y\in\Z)\ with\ x\eq1\pmod{4},\vspace{2mm}\\
\displaystyle\f{(-1)^{(p+1)/4}2p}{\binom{(p+1)/2}{(p+1)/4}}\ &if\ p\eq3\pmod{4},\end{cases}
\end{align}
and
\begin{align}\label{coreq2}
&\sum_{k=0}^{p-1}\f{T_k(2,-1)^2}{(-8)^k}\notag\\
\eq&\begin{cases}\displaystyle 2x-\f{p}{2x}\ &if\ p\eq1\pmod{4}\ and\ p=x^2+4y^2\ (x,y\in\Z)\ with\ x\eq1\pmod{4},\vspace{2mm}\\
\displaystyle 0\ &if\ p\eq3\pmod{4}.\end{cases}
\end{align}
If $p>3$, then we have
\begin{equation}\label{coreq3}
\sum_{k=0}^{p-1}\f{T_k(4,1)^2}{4^k}\eq\sum_{k=0}^{p-1}\f{T_k(4,1)^2}{36^k}\eq\l(\f{-1}{p}\r)\pmod{p^2}.
\end{equation}
\end{corollary}

Note that \eqref{coreq1} and \eqref{coreq3} confirm Conjecture \ref{sunconj} for the modulus $p^2$ cases. \eqref{coreq2} was brought to our attention by Sun and it was originally conjectured by Sun in his notebook where he has proved the modulus $p$ case. We also note that any prime $p\eq1\pmod{4}$ can be uniquely represented as $x^2+4y^2$, where $x,y\in\Z$ and $x\eq1\pmod{4}$ (cf. \cite[p. 106]{IR}). In \cite{Sun2014b}, Sun also showed that for any prime $p\nmid b-2c$, then
$$
\sum_{k=0}^{p-1}\f{T_k(b,c^2)^2}{(b-2c)^{2k}}\eq \l(\f{-c^2}{p}\r)\pmod{p}.
$$
Trivially, the above congruence holds for $c\eq0\pmod{p}$. If $c\not\equiv0\pmod{p}$, we can obtain the modulus $p^2$ extension of the congruence by \eqref{maintheq2}.

Our method to prove Theorem \ref{maintheorem} is quite different from the one used by Sun to prove \eqref{sunres}. We need the symbolic summation package \verb"Sigma" developed by Schneider \cite{S} to establish some identities. We shall prove Theorem \ref{maintheorem} and Corollary \ref{cor} in the next section.

\medskip

\section{Proofs of Theorem \ref{maintheorem} and Corollary \ref{cor}}
\setcounter{lemma}{0} \setcounter{theorem}{0}
\setcounter{equation}{0}\setcounter{proposition}{0}
\setcounter{Rem}{0}\setcounter{conjecture}{0}

To show the main results, we need the following preliminary results.

We need the following identity obtained by Sun \cite[Lemma 4.1]{Sun2014b} which is deduced from the well-known Clausen identity (cf. \cite[p. 116]{AAR99}).

\begin{lemma}\label{lem1}
Let $b,c\in\mathbb{Z}$ and $d=b^2-4c$. For any $n\in\mathbb{N}$ we have
\begin{equation}\label{lem1eq}
T_n(b,c)^2=\sum_{k=0}^n\binom{n+k}{2k}\binom{2k}{k}^2c^kd^{n-k}.
\end{equation}
\end{lemma}

For any $n\in\N$, define the $n$th harmonic $H_n:=\sum_{k=1}^n1/k$. Harmonic numbers have many congruence properties. For example, for any odd prime we have
$$
H_{p-1}\eq0\pmod{p}.
$$
(See \cite[p. 37]{IR}.) Note that the above congruence was extended to the modulus $p^2$ case by Wolstenholme \cite{Wolstenholme} for $p>3$.

We now give an identity involving harmonic numbers.

\begin{lemma}\label{lem2}
For any nonnegative integer $n$ we have
\begin{equation}\label{lem2eq}
\sum_{k=0}^{n}\binom{n}{k}H_kx^k=(1+x)^nH_n-\sum_{k=1}^n\f{(1+x)^{n-k}}{k}.
\end{equation}
\end{lemma}

\begin{proof}
We prove this identity by the symbolic summation package \verb"Sigma" in Mathematica. Denote the left-hand side of \eqref{lem2eq} by $S_n$. Using the command \verb"GenerateRecurrence" in \verb"Sigma" we find $S_n$ satisfies the following recurrence relation:
$$
-(n+1)(x+1)^2S_n+(2n+3)(x+1)S_{n+1}-(n+2)S_{n+2}=-x\ (n\in\N).
$$
It can be also verified that the right-hand side of \eqref{lem2eq} satisfies the same recurrence relation. Moreover, it is easy to see that the both sides of \eqref{lem2eq} coincide for $n=0,1$. Thus the desired result follows.
\end{proof}

\begin{Rem}
Taking $x=-1$ in \eqref{lem2eq} we obtain the following known identity (cf. \cite[(1.45)]{Gould}):
$$
\sum_{k=0}^n\binom{n}{k}(-1)^kH_k=-H_n.
$$
\end{Rem}

From \cite[(2.1) and (2.4)]{Gould} we know for any $n\in\N$ and $x\neq-1$,
\begin{equation}\label{knowneq}
\sum_{k=0}^{n}\f{x^k}{\binom{n}{k}}=(n+1)\sum_{k=1}^{n+1}\f{x^k+1}{k(x+1)}\l(\f{x}{x+1}\r)^{n+1-k}.
\end{equation}
The next lemma is very similar to \eqref{knowneq}.

\begin{lemma}\label{lem3}
For any $n\in\Z^{+}$ and $x\neq-1$, we have
\begin{equation}\label{lem3eq}
\sum_{k=0}^{n-1}\frac{x^k}{\binom{2n-1}{k}}=\f{2n}{x+1}\sum_{k=1}^{2n}\f{1}{k}\l(\f{x}{x+1}\r)^{2n-k}+\f{2n(x-1)}{(x+1)^2}\sum_{k=1}^n\f{x^k}{k\binom{2k}{k}}\l(\f{x}{x+1}\r)^{2n-2k}.
\end{equation}
\end{lemma}

\begin{proof}
In fact, by \verb"Sigma" we originally find that
\begin{align}\label{key}
\sum_{k=0}^{n-1}\f{x^k}{\binom{2n-1}{k}}=&\f{n}{(x+1)^2}\sum_{k=1}^{n}\f{(4k-1)x+2k-1}{k(2k-1)}\l(\f{x}{x+1}\r)^{2n-2k}\notag\\
&+\f{2n(x-1)}{(x+1)^2}\sum_{k=1}^n\f{x^k}{k\binom{2k}{k}}\l(\f{x}{x+1}\r)^{2n-2k}.
\end{align}
It is easy to verify that the both sides of \eqref{key} satisfy the following recurrence relation:
$$
(n+1)x^2S_n-n(1+x)^2S_{n+1}=-\f{n((4n+3)x+2n+1)}{2n+1}+\f{(n^2+n)(x^{n+1}-x^{n+2})}{(2n+1)\binom{2n}{n}}.
$$
Since the both sides of \eqref{key} coincide for $n=1$. Thus \eqref{key} holds. Now it suffices to show
\begin{equation}\label{key'}
\f{n}{(x+1)^2}\sum_{k=1}^{n}\f{(4k-1)x+2k-1}{k(2k-1)}\l(\f{x}{x+1}\r)^{2n-2k}=\f{2n}{x+1}\sum_{k=1}^{2n}\f{1}{k}\l(\f{x}{x+1}\r)^{2n-k}.
\end{equation}
Clearly,
\begin{align*}
&\f{n}{(x+1)^2}\sum_{k=1}^{n}\f{(4k-1)x+2k-1}{k(2k-1)}\l(\f{x}{x+1}\r)^{2n-2k}\\
=&\f{n}{(x+1)^2}\sum_{k=1}^n\l(\f{2x}{2k-1}+\f{x+1}{k}\r)\l(\f{x}{x+1}\r)^{2n-2k}\\
=&\f{2nx^{2n}}{(x+1)^{2n+1}}\l(\sum_{k=1}^n\f{1}{2k-1}\l(\f{x}{x+1}\r)^{1-2k}+\sum_{k=1}^n\f{1}{2k}\l(\f{x}{x+1}\r)^{-2k}\r)\\
=&\f{2n}{x+1}\sum_{k=1}^{2n}\f{1}{k}\l(\f{x}{x+1}\r)^{2n-k}.
\end{align*}
This completes the proof.
\end{proof}

\begin{lemma}\label{lem4}
For any $n\in\Z^{+}$ and $dm(m+d)\neq0$, we have
\begin{align}\label{lem4eq}
&\sum_{l=1}^{n}\binom{n}{l}\binom{n+l}{l}(-1)^l
\sum_{k=1}^{l}\frac {1}{k\binom{2k}{k}}\bigg(-\frac{(m-d)^2}{md}\bigg)^k\notag\\
=&\frac{d-m}{d+m}(-1)^{n}\bigg(\sum_{k=1}^{n}\frac{(-d/m)^k}{k}-
\sum_{k=1}^{n}\frac{(-m/d)^k}{k}\bigg).
\end{align}
\end{lemma}

\begin{proof}
\eqref{lem4eq} is also found by \verb"Sigma". It is easy to see that both sides of \eqref{lem4eq} satisfy the same recurrence relation:
\begin{align*}
&md(n+1)S_n+(-2d^2+md-2m^2-d^2n+mdn-m^2n)S_{n+1}\\
+&(-2d^2+3md-2m^2-d^2n+mdn-m^2n)S_{n+2}+md(n+3)S_{n+3}=0.
\end{align*}
Moreover, the both sides of \eqref{lem4eq} coincide when $n=0,1,2$. This proves the desired lemma.
\end{proof}

\begin{Rem}
If we replace the $-(m-d)^2/(md)$ by arbitrary $x\in\C$ in the left-hand side of \eqref{lem4eq}, then no simple form can be found by \verb"Sigma".
\end{Rem}

Let $p$ be a prime and $x$ a $p$-adic integer. Define the finite polylogarithms (cf. \cite{MT2013}) as follows:
$$
\pounds_{d}(x)=\sum_{k=1}^{p-1}\f{x^k}{k^d}.
$$
Let $Q_p(x)=(x^p+(1-x)^p-1)/p$. By Fermat's little theorem, $Q_p(x)\in\Z_p$. The following lemma concerns the congruence relation between $\pounds_{d}(x)$ and $Q_p(x)$.

\begin{lemma}[Mattarei and Tauraso \cite{MT2013}]\label{lem5} For any odd prime $p$ and $p$-adic integer $x$ with $p\nmid x(1-x)$ we have
$$
\pounds_1(x)\eq-Q_p(x)\pmod{p}.
$$
\end{lemma}

\begin{lemma}\label{lem6}
Let $p$ be an odd prime. Then for any $p$-adic integer $x$ we have
\begin{equation}\label{lem6eq}
\sum_{k=1}^{(p-1)/2}\f{(1-x)^k}{k}\eq H_{(p-1)/2}+S_{p-1}\l(\f{x}{4}\r)\pmod{p}.
\end{equation}
\end{lemma}

\begin{proof}
By the binomial theorem, we have
\begin{align*}
&\sum_{k=1}^{(p-1)/2}\f{(1-x)^k}{k}=\sum_{k=1}^{(p-1)/2}\f{1}{k}\sum_{s=0}^k\binom{k}{s}(-x)^s\\
=&H_{(p-1)/2}+\sum_{s=1}^{(p-1)/2}(-x)^s\sum_{k=s}^{(p-1)/2}\f{1}{k}\binom{k}{s}\\
=&H_{(p-1)/2}+\sum_{s=1}^{(p-1)/2}\f{(-x)^s}{s}\sum_{k=s}^{(p-1)/2}\binom{k-1}{s-1}\\
=&H_{(p-1)/2}+\sum_{s=1}^{(p-1)/2}\f{(-x)^s}{s}\binom{(p-1)/2}{s}\\
\eq&H_{(p-1)/2}+S_{p-1}\l(\f{x}4\r)\pmod{p},
\end{align*}
where we have used the Chu identity (cf. \cite[(1.51)]{Gould}) and the fact $\binom{(p-1)/2}{s}\eq\binom{2s}{s}/(-4)^s\pmod{p}$.
\end{proof}

Now we are in a position to prove Theorem \ref{maintheorem}.

\medskip

\noindent{\it Proof of Theorem \ref{maintheorem}}. Suppose that $p\mid md(m-d)$. Note that
$$
\binom{2k}{k}\eq0\pmod{p}
$$
for $k$ among $(p+1)/2,\ldots,p-1$. Therefore, in view of Lemma \ref{lem1} we have
\begin{align*}
\sum_{k=0}^{p-1}\frac{T_k(b,c)^2}{m^k}
&=\sum_{k=0}^{p-1}\frac{d^k}{m^k}\sum_{l=0}^k\binom{k+l}{2l}\binom{2l}{l}^2\f{c^l}{d^l}\\
&\eq\sum_{l=0}^{(p-1)/2}\binom{2l}{l}^2\f{c^l}{d^l}\sum_{k=l}^{p-1}\binom{k+l}{2l}\frac{d^k}{m^k}\\
&=\sum_{l=0}^{(p-1)/2}\binom{2l}{l}^2\f{c^l}{m^l}\sum_{k=0}^{p-l-1}\binom{k+2l}{2l}\frac{d^k}{m^k}\\
&=\Sigma_1+\Sigma_2 \pmod{p^2},
\end{align*}
where
$$
\Sigma_1:=\sum_{l=0}^{(p-1)/2}\binom{2l}{l}^2\f{c^l}{m^l}\sum_{k=0}^{p-2l-1}\binom{k+2l}{2l}\f{d^k}{m^k}
$$
and
$$
\Sigma_2:=\sum_{l=1}^{(p-1)/2}\binom{2l}{l}^2\f{c^l}{m^l}\sum_{k=p-2l}^{p-l-1}\binom{k+2l}{2l}\f{d^k}{m^k}.
$$

We first evaluate $\Sigma_1$ modulo $p^2$. Clearly, for any $k\in\{0,1,\ldots,p-2l-1\}$ we have
$$
\binom{p-2l-1}{k}=\prod_{j=1}^k\f{p-2l-j}{j}\eq\binom{-2l-1}{k}(1-p(H_{2l+k}-H_{2l}))\pmod{p^2}.
$$
Hence we have
\begin{align}\label{sig1}
\Sigma_1=&\sum_{l=0}^{(p-1)/2}\binom{2l}{l}^2\f{c^l}{m^l}\sum_{k=0}^{p-2l-1}\binom{-2l-1}{k}\l(-\f{d}{m}\r)^k\notag\\
\eq&\sum_{l=0}^{(p-1)/2}\binom{2l}{l}^2\f{c^l}{m^l}\sum_{k=0}^{p-2l-1}\binom{p-2l-1}{k}\l(-\f{d}{m}\r)^k(1+p(H_{2l+k}-H_{2l}))\notag\\
=&\sum_{l=0}^{(p-1)/2}\binom{2l}{l}^2\f{c^l}{m^l}\l(\f{m-d}{m}\r)^{p-2l-1}(1-pH_{2l})\notag\\
&+p\sum_{l=0}^{(p-1)/2}\binom{2l}{l}^2\f{c^l}{m^l}\sum_{k=0}^{p-2l-1}\binom{p-2l-1}{k}\l(-\f{d}{m}\r)^kH_{2l+k}\pmod{p^2}.
\end{align}
Note that for any $0\leq k\leq p-1$,
$$
H_{p-1-k}=\sum_{j=1}^{p-1-k}\f{1}{j}=\sum_{j=k+1}^{p-1}\f{1}{p-j}\eq H_k-H_{p-1}\eq H_k\pmod{p}.
$$
Thus
\begin{align*}
&\sum_{k=0}^{p-2l-1}\binom{p-2l-1}{k}\l(-\f{d}{m}\r)^kH_{2l+k}\\
=&\sum_{k=0}^{p-2l-1}\binom{p-2l-1}{k}\l(-\f{d}{m}\r)^{p-2l-1-k}H_{p-1-k}\\
\eq&\l(-\f{m}{d}\r)^{2l}\sum_{k=0}^{p-2l-1}\binom{p-2l-1}{k}\l(-\f{m}{d}\r)^kH_k\\
\eq&\l(\f{m}{m-d}\r)^{2l}\l(H_{2l}-\sum_{k=1}^{p-2l-1}\f{1}{k}\l(\f{d}{d-m}\r)^k\r)\\
\eq&\l(\f{m}{m-d}\r)^{2l}H_{2l}+\f{d}{d-m}\l(\f{m}{m-d}\r)^{2l}\sum_{k=2l+1}^{p-1}\f{1}{k}\l(\f{d-m}{d}\r)^k\pmod{p},
\end{align*}
where we have used Lemma \ref{lem2} with $x=-m/d$. Substituting this into \eqref{sig1} we arrive at
\begin{align}\label{sig1'}
\Sigma_1\eq& \l(\f{m-d}{m}\r)^{p-1}\sum_{l=0}^{(p-1)/2}\binom{2l}{l}^2\l(\f{cm}{(m-d)^2}\r)^l\notag\\
&+\f{pd}{d-m}\sum_{l=0}^{(p-1)/2}\binom{2l}{l}^2\l(\f{cm}{(m-d)^2}\r)^l\sum_{k=2l+1}^{p-1}\f{1}{k}\l(\f{d-m}{d}\r)^k\pmod{p^2}.
\end{align}

Now we turn to consider $\Sigma_2$ modulo $p^2$. Note that for $l\in\{1,2,\ldots,(p-1)/2\}$ we have
\begin{align}
\binom{p+k}{2l}&=\frac{(p+k)\cdots(p+1) p(p-1)\cdots(p+k-2l+1)}{(2l)!}\notag\\
&\equiv (-1)^{k+1}\frac{p}{2l}\frac{k!(2l-1-k)!}{(2l-1)!}\notag=\frac{p}{2l}\frac{(-1)^{k+1}}{\binom{2l-1}{k}}\pmod{p^2}.
\end{align}
So using Lemma \ref{lem3} with $x=-d/m$ we have
\begin{align}\label{sig2}
\Sigma_2=&\sum_{l=1}^{(p-1)/2}\binom{2l}{l}^2\f{c^l}{m^l}\sum_{k=p-2l}^{p-1}\binom{k+2l}{2l}\f{d^k}{m^k}=\sum_{l=1}^{(p-1)/2}\binom{2l}{l}^2\f{c^l}{m^l}\sum_{k=0}^{l-1}\binom{p+k}{2l}\l(\f{d}{m}\r)^{p+k-2l}\notag\\
\eq&-\f{pd}{2m}\sum_{l=1}^{(p-1)/2}\f{\binom{2l}{l}^2}{l}\l(\f{cm}{d^2}\r)^l\sum_{k=0}^{l-1}\l(-\f{d}{m}\r)^k\f{1}{\binom{2l-1}{k}}\notag\\
\eq&\f{pd}{d-m}\sum_{l=1}^{(p-1)/2}\binom{2l}{l}^2\l(\f{cm}{(m-d)^2}\r)^l\sum_{k=1}^{2l}\f{1}{k}\l(\f{d-m}{d}\r)^k\notag\\
&+\f{pd(m+d)}{(d-m)^2}\sum_{l=1}^{(p-1)/2}\binom{2l}{l}^2\l(\f{cm}{(m-d)^2}\r)^l\sum_{k=1}^l\f{1}{k\binom{2k}{k}}\l(-\f{(m-d)^2}{md}\r)^k\pmod{p^2}.
\end{align}

If $m=-d$, then combining \eqref{sig1'} and \eqref{sig2} we have
\begin{align*}
&\sum_{k=0}^{p-1}\f{T_k(b,c)^2}{m^k}\eq\Sigma_1+\Sigma_2\\
\eq&2^{p-1}\sum_{l=0}^{(p-1)/2}\binom{2l}{l}^2\l(-\f{c}{4d}\r)^l+\f12p\sum_{l=0}^{(p-1)/2}\binom{2l}{l}^2\l(-\f{c}{4d}\r)^l\pounds_1(2)\\
\eq&\sum_{l=0}^{(p-1)/2}\binom{2l}{l}^2\l(-\f{c}{4d}\r)^l\pmod{p^2},
\end{align*}
where in the last step follows from Lemma \ref{lem5}. This proves \eqref{maintheq1}.

Below we assume that $16mc=(m-d)^2$. Now from \eqref{sig1'} and \eqref{sig2} we have
\begin{align}\label{sig1+sig2}
\Sigma_1+\Sigma_2\eq&\sum_{l=0}^{(p-1)/2}\f{\binom{2l}{l}^2}{16^l}\l(\l(\f{m-d}{m}\r)^{p-1}+\f{pd}{d-m}\pounds_1\l(\f{d-m}{d}\r)\r)\notag\\
&+\f{pd(m+d)}{16mc}\sum_{l=1}^{(p-1)/2}\f{\binom{2l}{l}^2}{16^l}\sum_{k=1}^{l}\f{1}{k\binom{2k}{k}}\l(-\f{(m-d)^2}{md}\r)^k\pmod{p^2}.
\end{align}
It is known from \cite[Lemma 3.1]{Sun2011} that for any $0\leq l\leq (p-1)/2$ we have
$$
\binom{(p-1)/2}{l}\binom{(p-1)/2+l}{l}(-1)^k\eq \f{\binom{2l}{l}^2}{16^l}\pmod{p^2}.
$$
Therefore by Lemmas \ref{lem4} and \ref{lem6} we arrive at
\begin{align}\label{important}
&\sum_{l=1}^{(p-1)/2}\f{\binom{2l}{l}^2}{16^l}\sum_{k=1}^{l}\f{1}{k\binom{2k}{k}}\l(-\f{16c}{d}\r)^k\notag\\
\eq&\sum_{l=1}^{(p-1)/2}\binom{(p-1)/2}{l}\binom{(p-1)/2+l}{l}(-1)^k\sum_{k=1}^{l}\f{1}{k\binom{2k}{k}}\l(-\f{(m-d)^2}{md}\r)^k\notag\\
\eq&\f{d-m}{d+m}\l(\f{-1}{p}\r)\l(\sum_{k=1}^{(p-1)/2}\f{(-d/m)^k}{k}-\sum_{k=1}^{(p-1)/2}\f{(-m/d)^k}{k}\r)\notag\\
\eq&\f{d-m}{d+m}\l(\f{-1}{p}\r)\l(S_{p-1}\l(\f{m+d}{4m}\r)-S_{p-1}\l(\f{m+d}{4d}\r)\r)\pmod{p}.
\end{align}
Mortenson \cite{Mortenson} proved that for any odd prime $p$ we have
\begin{equation}\label{hyper}
\sum_{k=0}^{p-1}\f{\binom{2k}{k}^2}{16^k}\eq\l(\f{-1}{p}\r)\pmod{p^2}
\end{equation}
as conjectured by Rodriguez-Villegas \cite{RV}. Substituting \eqref{important} and \eqref{hyper} into \eqref{sig1+sig2} and using Lemma \ref{lem5}, we immediately obtain \eqref{maintheq2}.

The proof of Theorem \ref{maintheorem} is now complete.\qed

\medskip

\noindent{\it Proof of Corollary \ref{cor}}. We first consider \eqref{coreq1} and \eqref{coreq2}. By \eqref{maintheq1}, we have
$$
\sum_{k=0}^{p-1}\f{T_k(2,2)^2}{4^k}\eq\sum_{k=0}^{p-1}\f{\binom{2k}{k}^2}{8^k}\pmod{p^2}
$$
and
$$
\sum_{k=0}^{p-1}\f{T_k(2,2)^2}{(-8)^k}\eq\sum_{k=0}^{p-1}\f{\binom{2k}{k}^2}{32^k}\pmod{p^2}.
$$
Sun \cite{Sun2011a} conjectured that if $p\eq1\pmod{4}$ and $p=x^2+4y^2$ with $x\eq1\pmod{4}$, then 
$$
\l(\f{2}{p}\r)\sum_{k=0}^{p-1}\f{\binom{2k}{k}^2}{8^k}\eq \sum_{k=0}^{p-1}\f{\binom{2k}{k}^2}{32^k}\eq 2x-\f{p}{2x}.
$$
Later, Z.-H. Sun \cite{SunZH2011} confirmed this conjecture and showed that
$$
\sum_{k=0}^{p-1}\f{\binom{2k}{k}^2}{32^k}\eq0\pmod{p^2}
$$
for $p\eq3\pmod{4}$. Sun \cite{Sun2013} proved that if $p\eq3\pmod{4}$, then
$$
\sum_{k=0}^{p-1}\f{\binom{2k}{k}^2}{8^k}\eq\f{(-1)^{(p+1)/4}2p}{\binom{(p+1)/2}{(p+1)/4}}\pmod{p^2}.
$$
In view of the above, we obtained \eqref{coreq1} and \eqref{coreq2} as desired.

Now we consider \eqref{coreq3}. Taking $b=4,c=1$ we have $d=b^2-4c=12$. Solving the equation $(m-12)^2=16m$ we obtain $m=4,36$. Therefore, \eqref{coreq3} is a special case of \eqref{maintheq2}. By \eqref{maintheq2} we deduce that
$$
\sum_{k=0}^{p-1}\f{T_k(4,1)^2}{4^k}\eq \l(\f{-1}{p}\r)+\f{3p}{2}\l(\f{-1}{p}\r)\l(q_p(12)-q_p(4)+S_{p-1}(1)-S_{p-1}\l(\f13\r)\r)\pmod{p^2}.
$$
For any prime $p>3$, Pan and Sun \cite{PanSun} proved that $S_{p-1}(1)\eq0\pmod{p}$; Sun and Tauraso \cite{Sun2010} proved that $S_{p-1}(1/3)\eq q_p(3)\pmod{p}$. Combining the above, we obtain \eqref{coreq3} at once.\qed

\begin{Acks}
The authors are grateful to Prof. Zhi-Wei Sun for his helpful suggestions on this paper. This work is supported by the National Natural Science Foundation of China (grant no. 11971222).
\end{Acks}

\end{document}